\documentclass[12pt,a4paper,reqno]{amsart}

\usepackage{graphicx}
\usepackage{amsfonts,amsthm,latexsym,amsmath,amssymb,amscd,epsfig}
\usepackage{graphics,graphicx}
\usepackage[all]{xy}
\usepackage{young}

\usepackage{tikz}
\usetikzlibrary{matrix}

\newtheorem{theorem}{Theorem}[section]
\newtheorem{lemma}[theorem]{Lemma}

\theoremstyle{definition}

\newtheorem{example}[theorem]{Example}

\theoremstyle{remark}

\numberwithin{equation}{section}

\newcommand{\nc}{\newcommand}
\renewcommand{\frak}{\mathfrak}
\providecommand{\cal}{\mathcal}
\renewcommand{\bold}{\mathbf}

\nc \Ab{{\ensuremath{\bold A}}}
\nc \ab{{\ensuremath{\bold a}}}
\nc \bb{{\ensuremath{\bold b}}}
\nc \cb{{\ensuremath{\bold c}}}
\nc \Bb{{\ensuremath{\bold B}}}
\nc \Gb{{\ensuremath{\bold G}}}
\nc \Qb{{\ensuremath{\bold Q}}}
\nc \Rb{{\ensuremath{\bold R}}} \nc \Cb{{\ensuremath{\bold C}}} 
\nc \Eb{{\ensuremath{\bold E}}}
\nc \eb{{\ensuremath{\bold e}}}
\nc \Db{{\ensuremath{\bold D}}}
\nc \Fb{{\ensuremath{\bold F}}}
\nc \ib{{\ensuremath{\bold i}}}
\nc \jb{{\ensuremath{\bold j}}}
\nc \kb{{\ensuremath{\bold k}}}
\nc \nb{{\ensuremath{\bold n}}}
\nc \rb{{\ensuremath{\bold r}}}
\nc \Pb{{\ensuremath{\bold P}}}
\nc \pb{{\ensuremath{\bold p}}}
\nc \SPb{{\ensuremath{\bold {SP}}}}
\nc \Zb{{\ensuremath{\bold Z}}} 
\nc \zb{{\ensuremath{\bold z}}} 
\nc \gb{{\ensuremath{\bold g}}} 
\nc \fb{{\ensuremath{\bold f}}} 
\nc \ub{{\ensuremath{\bold u}}} 
\nc \vb{{\ensuremath{\bold v}}} 
\nc \yb{{\ensuremath{\bold y}}} 
\nc \xb{{\ensuremath{\bold x}}} 
\nc \xib{{\ensuremath{\bold \xi}}} 
\nc \Nb{{\ensuremath{\bold N}}} 
\nc \Hb{{\ensuremath{\bold H}}} 
\nc \wb{{\ensuremath{\bold w}}} 
\nc \Wb{{\ensuremath{\bold W}}} 
\nc \syz{{\mathbf {syz}}}
\nc \bnoll{{\ensuremath{\bold 0}}} 

\nc \mf{\frak m} \nc \mh{\hat{\m}} 
\nc \nf{\frak n}
\nc \Of{\frak O}
\nc \rf{\frak r}
\nc \mufr{{\mathbf \mu}}
\nc \hf{\frak h} 
\nc \qf{\frak q} 
\nc \bfr{\frak b} 
\nc \kfr{\frak k} 
\nc \pfr{\frak p} 
\nc \af{\frak a }
\nc \cf{\frak c }
\nc \sfr{\frak s} 
\nc \ufr{\frak u} 
\nc \g{\frak g} 
\nc \gA{\g_{\Ao}} 
\nc \lfr{\frak l}
\nc \afr{\frak a}
\nc \gfh{\hat {\frak g}}
\nc \gl{\frak { gl }}
\nc \Sl{\frak {sl}}
\nc \SU{\frak {SU}}
\nc{\Homf}{\frak{Hom}}

\newcommand{\on}{\operatorname}
\nc\hankel{\on {Hankel}}
\nc\row{\on {row\ }}
\nc\nullity{\on {nullity }}
\nc\col{\on {col\ }}
\nc\rowm{\on {Row \ }}
\nc\loc{\on {lc \ }}
\nc\nullo{\on {null\ }}
\nc\Nul{\on {Nul\ }}
\nc \Ann {\on {Ann }}
\nc \Ass {\on {Ass \ }}
\nc \Coker {\on {Coker}}
\nc \Co{\on C}
\nc \Homo{\on {Hom}}
\nc \Ker {\on {Ker}}
\nc \omod{\on {mod}}
\nc \No {\on N}
\nc \NGo {\on {NG}}
\nc \Oo {\on O}
\nc \ch {\on {ch}}
\nc \rko {\on {rk}}
\nc \Sing {\on {Sing\ }}
\nc \Reg {\on {Reg}}
\nc \CoI {\on {CI}}
\nc \CoM {\on {CM}}
\nc \Gor {\on {Gor}}
\nc \Type {\on {Type}}
\nc \can {\on {can}}
\nc \Top {\on {T}}
\nc \rel {\on {rel}}
\nc \sgn {\on {sgn }}
\nc \trdeg {\on {tr.deg}}
\nc \codim {\on {codim }}
\nc \coht {\on {coht}}
\nc \divo {\on {div \ }}
\nc \coh {\on {coh}}
\nc \Clo {\on {Cl}}
\nc \embdim{\on {embdim}}
\nc \embcodim{\on {embcodim \ }}
\nc \qcoh {\on {qcoh}}
\nc \grad {\on {grad}\ }
\nc \grade {\on {grade}}
\nc \hto {\on {ht}}
\nc \depth {\on {depth}}
\nc \prof {\on {prof}}
\nc \reso{\on {res}}
\nc \ind{\on {ind}}
\nc \prodo{\on {prod}}
\nc \coind{\on {coind}}
\nc \Con{\on {Con}}
\nc \Crit{\on {Crit}}
\nc \Der{\on {Der}}
\nc \Char{\on {Char}}
\nc \Ch{\on {Ch}}

\nc \Ext{\on {Ext}}
\nc \Eo{\on {E}}
\nc \End{\on {End}}
\nc \Ad{\on {Ad}}
\nc \gr{\on {gr}}
\nc \Fo{\on {F}}
\nc \Gr{\on {Gr}}
\nc \Go{\on {G}}
\nc \GFo{\on {GF}}
\nc \Glo{\on {Gl}}
\nc \Ho{\on {H}}
\nc \CMo{\on {\CM}}
\nc \SCM{\on {SCM}}
\nc \hol{\on {hol}}
\nc{\sgd}{\on{sgd}}
\nc \supp{\on {supp}}
\nc \ssupp{\on {s-supp}}
\nc \singsupp{\on {singsupp}}
\nc \msupp{\on {msupp}}
\nc \spec{\on {spec}}
\nc \spano{\on {span }}
\nc \Span{\on {Span }}
\nc \Max{\on {Max}}
\nc \Min{\on {Min}}
\nc \Mod{\on {Mod}}
\nc \Rad {\on {Rad}}
\nc \rad {\on {rad}}
\nc \rank {\on {rank}}
\nc \range {\on {range}}
\nc \Slo{\on {SL}}
\nc \soc {\on {soc}}
\nc \Irr {\on {Irr}}
\nc \Imo {\on {Im}}
\nc \SSo{\on {SS}}
\nc \lub{\on {lub}}
\nc \gldim{\on {gl.d.}}
\nc \pdo{\on {p.d.}} 
\nc \ido{\on {i.d.}} 
\nc \dSSo{\dot {\SSo}}
\nc \So{\on S}
\nc \Io{\on I}
\nc \Jo{\on J}
\nc \jo{\on j}
\nc \Ko{\on K}
\nc \PBW{\Ac_{PBW}}
\nc \Ro{\on R}
\nc \To{\on T}
\nc \Ao{\on A}

\nc \Do{{\on D}}
\nc \Bo{\on B}
\nc \Po{\on P}
\nc \Qo{\on Q}
\nc \Zo{\on Z}
\nc \U{\on U}
\nc \wt{\on {wt}}
\nc \Uh{\hat {\U}}
\nc \T{\on T}
\nc \Lo{\on L}
\nc{\dop}{\on d}
\nc{\eo}{\on e}
\nc{\ado}{\on{ad}}
\nc{\Tot}{\on{Tot}}
\nc{\Aut}{\on{Aut}}
\nc{\sinc}{\on {sinc}}

\nc{\overrightleftarrows}[2]{\overset{#1}{\underset{#2}{\rightleftarrows}}}

\nc{\CCF}{\cal{CF}}

\nc{\CDF}{\cal{DF}}
\nc{\CHC}{\check{\cal C}}

\nc{\Cone}{\on{Cone}}
\nc{\dec}{\on{dec}}

\nc{\Diff}{\on{Diff}}

\nc{\dirlim}{\underset{\to}{\on{lim}}}
\nc{\dpar}{\partial}
\nc{\GL}{\on{GL}}
\nc{\CGr}{\cal{G}r}

\nc{\pr}{\on{pr}}
\nc{\semid}{|\!\!\!\times}
\nc{\Hom}{\on{Hom}}
\nc \RHom{\on {RHom}}

\nc \Proj{\mathrm {Proj\ }}
\nc \proj{\mathrm {proj}}
\nc{\Id}{\on{Id}}
\nc{\id}{\on{id}}
\nc{\Ima}{\on{Im}}
\nc{\invtimes}{\underset{\gets}{\otimes}}
\nc{\invlim}{\underset{\gets}{\on{lim}}}
\nc{\Lie}{\on{Lie}}

\nc{\re}{\on{Re }}
\nc{\Pic}{\on{Pic }}
\nc{\LPic}{\on{LPic }}
\nc{\Sch}{\on{Sch}}
\nc{\Sh}{\on{Sh}}
\nc{\Set}{\on{Set}}
\nc{\spo}{\on{sp\  }}
\nc{\Spec}{\on{Spec}}
\nc{\mSpec}{\on{mSpec}}
\nc{\Specb}{\bold {Spec}}
\nc{\Projb}{\bold {Proj}}
\nc{\Specan}{\on{Specan}}
\nc{\Spo}{\on{Sp}}
\nc{\Spf}{\on{Spf}}
\nc{\sym}{\on{sym}}
\nc{\symm}{\on{symm}}
\nc{\rop}{\on{r}}
\nc{\Td}{\on{Td}}
\nc{\Tor}{\on{Tor}}

\nc{\Artin}{\cal{A}rtin}
\nc{\Dgcoalg}{\cal{D}gcoalg}
\nc{\Dglie}{\cal{D}glie}
\nc{\Ens}{\cal{E}ns}
\nc{\Fsch}{\cal{F}sch}
\nc{\Groupoids}{\cal{G}roupoids}
\nc{\Holie}{\cal{H}olie}
\nc{\Mor}{\cal{M}or}

\nc{\CF}{\ensuremath{\cal{F}}}
\nc \Kc{\ensuremath{\cal K}}
\nc \lcc{{\mathcal l}} 
\nc \CC{{\ensuremath{\cal C}}} 
\nc \Cc{{\ensuremath {\cal C}}}
\nc \Pc{{\ensuremath{\cal P}}}
\nc \Dc{\ensuremath{\mathcal D}}
\nc \Ac{{\ensuremath{\cal A}}} 
\nc \Bc{{\ensuremath{\cal B}}}
\nc \Ec{{\ensuremath{\cal E}}}
\nc \Fc{{\ensuremath{\cal F}}}
\nc \Mcc{{\ensuremath{\cal M}}} 
\nc \hM{\hat{\Mcc}} 
\nc \bM{\bar {\Mcc}} 
\nc\hbM{\hat{\bar \Mcc}}  
\nc \Nc{{\ensuremath{\cal N}}}
\nc \Hc{{\ensuremath{\cal H}}} 
\nc \Ic{{\ensuremath{\cal I}}} 
\nc \Oc{\ensuremath{{\cal O}}}
\nc \Och{\hat{\cal O}} 
\nc \Sc{{\ensuremath{{\cal S}}}}
\nc \Tc{\ensuremath{{\cal T}}} 
\nc \Vc{{\ensuremath{{\cal V}}}} 
\nc{\CA}{{\ensuremath{{\cal A}}}}
\nc{\CB}{{\ensuremath{{\cal B}}}}
\nc{\C}{{\ensuremath{{\cal F}}}}
\nc{\Gc}{{\ensuremath{{\cal G}}}}
\nc{\CH}{\ensuremath{\mathcal H}}
\nc{\CI}{{\ensuremath{{\cal I}}}}
\nc{\CM}{{\ensuremath{{\cal M}}}}
\nc{\CN}{{\ensuremath{{\cal N}}}}
\nc{\CO}{{\ensuremath{{\cal O}}}}
\nc{\Rc}{{\ensuremath{{\cal R}}}}
\nc{\CT}{{\ensuremath{\mathcal T}}}
\nc{\CU}{\ensuremath{{\cal U}}}
\nc{\CV}{\ensuremath{{\cal V}}}
\nc{\CZ}{\ensuremath{{\cal Z}}}
\nc{\Homc}{\ensuremath{{\cal {Hom}}}}

\nc{\Tab}{\ensuremath{{\mbox{Tab}}}}

\nc{\STab}{\ensuremath{{\mbox{STab}}}}

\nc{\fa}{\frak{a}}
\nc{\fA}{\frak{A}}
\nc{\fg}{\frak{g}}
\nc{\fh}{\frak{h}}
\nc{\fI}{\frak{I}}
\nc{\fK}{\frak{K}}
\nc{\fm}{\frak{m}}
\nc{\fP}{\frak{P}}
\nc{\fS}{\frak{S}}
\nc{\ft}{\frak{t}}
\nc{\fX}{\frak{X}}
\nc{\fY}{\frak{Y}}

\nc{\bF}{\bar{F}}
\nc{\bCP}{\bar{\cal{P}}}
\nc{\bm}{\mbox{\bf{m}}}
\nc{\bT}{\mbox{\bf{T}}}
\nc{\bS}{\mbox{\bf{S}}}
\nc{\hB}{\hat{B}}
\nc{\hC}{\hat{C}}
\nc{\hP}{\hat{P}}
\nc{\htest}{\hat P}

\nc{\nen}{\newenvironment}
\nc{\ol}{\overline}
\nc{\ul}{\underline}
\nc{\ra}{\to}
\nc{\lla}{\longleftarrow}
\nc{\lra}{\longrightarrow}
\nc{\Lra}{\Longrightarrow}
\nc{\Lla}{\Longleftarrow}
\nc{\Llra}{\Longleftrightarrow}
\nc{\hra}{\hookrightarrow}
\nc{\iso}{\overset{\sim}{\lra}}

\nc{\dsize}{\displaystyle}
\nc{\sst}{\scriptstyle}
\nc{\tsize}{\textstyle}

\nen{exa}[1]{\label{#1}{\bf Example.\ } }{}

\nen{rem}[1]{\label{#1}{\em Remark.\ } }{}
\nen{exer}[1]{\label{#1}{\em Exercise.\ } }{}

\newcommand {\ZZ}{\mathbb{Z}}

\newcommand{\Res}{\mathrm{Res}}
\newcommand{\Disc}{\mathrm{Disc}}

\newcommand{\UU}{\mathbb{U}}

\begin{document}
\title{Resultant of an equivariant polynomial system with respect to direct product of symmetric groups}

\author{Sonagnon Julien Owolabi, Ibrahim Nonkan\'e, Joel Tossa}

\address{ Institut de Math\'ematiques et de Sciences Physiques, IMSP, Universit\'e d' Abomey-Calavi, B\'enin}
\email{julien.owolabi@imsp-auc.org, joel.tossa@imsp-uac.org},
\address{D\'epartement d'\'economie et de math\'ematiques appliqu\'ees, IUFIC, Universit\'e Thomas Sankara, Burkina faso}
\email{ibrahim.nonkane@uts.bf}


\maketitle
\begin{abstract}
In this note, we consider the resultant of systems of homogeneous multivariate polynomials which are equivariant under the action of  direct product of two symmetric groups. We establish a decomposition formula for the resultant of such systems. Thanks to that decomposition formula we prove that the discriminant of an invariant multivariate homogeneous polynomial under a direct product of symmetric groups  splits into smaller resultants that are easier to compute.
 
\textbf{keywords}:\ {commutative algebra, symbolic computation, resultant, discriminant, divided difference, direct product of symmetric groups}
\end{abstract}

\section{Motivation and Introduction}

Solving algebraic systems of  polynomial equations $f_1, f_2, \ldots,f_n$ in several variable is  a fundamental problem with in  computational algebra with many applications (cryptology, robotics, biology, physic, coding theory, etc...). The analysis of such systems is based on the study of the resultant \cite{Jou91}.  System which are invariant under the action of a group may be great importance since symmetry are very relevant in physical sciences as it has to with energy. Thus Laurent Bus\'e and Anna Karasoulou has studied the resultant of an equivariant polynomial system with respect to $\mathcal{S}_n$ group  of permutations  on a set of variables $ \{ x_1,\ldots,x_n \} $ \cite{BuKa16}. They develop a nice decomposition of that resultant which leads to the decomposition of the discriminant of a symmetric polynomial. In some situations the permutations among the set $x_1,\ldots,x_n$ may not be effective  in the sense that some action may hindered or neglected, and in this case the symmetric group would not be of the best description of the symmetry. For example the coordinates $x_1,\ldots,x_n$ of the particules of a given molecule may separated into two subsets $\{x_1, \ldots, x_p\}$ and $\{ x_{p+1}, \ldots, x_n\}$ which do not interact. The symmetry is therefore described by the direct product of symmetric groups $\mathcal{S}_p \times \mathcal{S}_{n-p},$ where $\mathcal{S}_p$ and $\mathcal{S}_{n-p}$  are groups of permutations on  $\{x_1, \ldots, x_p\}$ and   $\{ x_{p+1}, \ldots, x_n\}$ respectively. A similar situation may occur when the coordinates separated into three or more subsets, leading to a product of three or more symmetric groups. Therefore we think that some results of \cite{BuKa16} may be generalizes to systems which are equivariant with respect to the product of symmetric groups, even to other groups. In this paper, we attempt to study the resultant of  an equivariant system with respect  to the direct product of two subgroup of $\mathcal{S}_n.$ This paper partially expository, in the sense we realize that  the same techniques that have been used is the case of the symmetric groups \cite{BuKa16} work for the case for the direct product of symmetric groups, then we make great use of them in this paper. This paper is somehow a variant of \cite{BuKa16} and we mainly refer to it for the proofs.

A polynomial system $\mathcal{A}= \{ f_1, f_2, \ldots,f_n\}$ is said to be {\it equivariant with respect} to a finite group $G $ if for all $g \in G , f_i \in \mathcal{A}, g(f_i) \in \mathcal{A}, i= 1,\ldots,n.$ In other words $\mathcal{A}$ is globally stable under the finite group $G$.
Let a system of $n$ homogeneous polynomials  $f^{\{1\}},\ldots, f^{\{p\}}, f^{\{p+1\}},\ldots, f^{\{n\}}$ 
of same degree $d$ equivariant to the direct product $\mathcal{S}_{\{1
,\ldots,p\}}\times\mathcal{S}_{\{p+1,\ldots,n\}} $ of two symmetric subgroups of  $\mathcal{S}_{n}$ with $1\leq p< n$. the action of $\mathcal{S}_{\{1,\ldots,p\}}\times\mathcal{S}_{\{p+1,\ldots,n\}} $ on  $f^{\{1\}},\ldots, f^{\{p\}}, f^{\{p+1\}},\ldots, f^{\{n\}}$ is described as follows. Let $\sigma_1\in \mathcal{S}_{\{1,\ldots,p\}}$, and $\sigma_2\in \mathcal{S}_{\{p+1
,\ldots,n\}}$ we have for all  \ $ i= 1,\ldots,n $ \\
$$(\sigma_1,\sigma_2)\Big(f^{\{i\}} \Big)(x_1,\ldots,x_p,x_{p+1},\ldots,x_n)=f^{\{i\}}(x_{\sigma_1(1)},\ldots x_{\sigma_1(p)},x_{\sigma_2(p+1)},\ldots,x_{\sigma_2(n)}).$$
We assume that for all $k \in \{1,\ldots,p\}$ $\sigma= (\sigma_1, \sigma_2) \in \mathcal{S}_{\{1,\ldots,p\}}\times\mathcal{S}_{\{p+1,\ldots,n\}}$, 
\begin{equation}
\sigma \Big( f^{\{k\}} \Big)= \begin{cases}
\sigma_1 \Big( f^{\{k\}} \Big)= f^{\{ {\sigma_1(k)} \}}\  \mbox{if}\  k\in \{1,\ldots,p\} \\
\sigma_2 \Big( f^{\{k\}} \Big)= f^{\{ {\sigma_2(k)} \}}\  \mbox{if}\  k\in \{p+1,\ldots,n\}.
\end{cases}
\end{equation}
Under this assumption, the polynomial system is equivariant with respect to the direct product $\mathcal{S}_{\{1,\ldots,p\}}\times\mathcal{S}_{\{p+1,\ldots,n\}}.$ In what follows, we set 
\begin{equation}\label{eq:GG}
 f^{\{\sigma(i)\}} := \sigma \Big( f^{\{i\}} \Big), \  \sigma \in \mathcal{S}_{\{1,\ldots,p\}}\times\mathcal{S}_{\{p+1,\ldots,n\}}, i=1,\ldots,n.
 \end{equation}
In this work, we will study the resultant of such systems. As an application, we obtain a decomposition formula for the discriminant of an invariant multivariate homogeneous polynomial under the action of a direct product of two symmetric groups.

\section{Resultant of a $\mathcal{S}_{\{1,\ldots,p\}}\times \mathcal{S}_{\{p+1,\ldots,n\}}$-equivariant polynomial system}\label{sec:resultant}
Let $R$ be commutative ring and denote by $R[x_1,\ldots,x_n]$ the ring of polynomials in $n \geq 2$ variables which is graded with the usual weights: $\deg (x_i)=1$ for all $i \in \{ 1,\ldots,n\}.$
In this section, we consider a polynomial system of $n$ homogeneous polynomials $f^{\{1\}},\ldots, f^{\{p\}}, f^{\{p+1\}},\ldots, f^{\{n\}}$ in $R[x_1,\ldots,x_n]$ of same degree $d$ which is equivariant to the direct product $\mathcal{S}_{\{1
,\ldots,p\}}\times\mathcal{S}_{\{p+1
,\ldots,n\}} $ of two subgroups of  $\mathcal{S}_{n}$ with $1\leq p< n$.

\subsection{Partitions}

Let $\lambda:=(\lambda_1,\lambda_2,\ldots,\lambda_{r_1})$ be a sequence such that $\lambda_1\geq \cdots \geq \lambda_{r_1} > 0$. When $\sum_{i=1}^{r_1}\lambda_i=p$, we will say such a $\lambda$ is a partition of $p$, and write $\lambda\vdash p.$

Given a partitions $\lambda  \vdash p$ its associated multinomial coefficient is defined as the integer
\begin{equation}\label{eq:mlambda}
m_{\lambda}:=\frac{1}{\prod_{j=1}^{n}s_{j}!} {{p}\choose{\lambda_1, \lambda_2,\ \ldots , \lambda_{r_1}}}= \frac{ p!}{(\prod_{j=1}^{p}s_{j}!) \lambda_1 !\lambda_2! \cdots \lambda_{r_1}!} .
\end{equation}
where $s_j$ denotes the number of boxes having exactly $j$ objects,   $j\in [p]$ for the partition $\lambda \vdash p.$

Let $\Lambda=(\lambda,\lambda')$ be a couple of partition swhere $\lambda=(\lambda_1,\lambda_2,\ldots,\lambda_{r_1}) \vdash p$ and $\lambda'=(\lambda'_1,\lambda'_2,\ldots,\lambda'_{r_2}) \vdash q$, the  we will write  $(\lambda,\lambda')\vdash (p,q)$ or $ \Lambda \vdash (p,q)$.
Given a couple of partitions $\Lambda=(\lambda,\lambda')\vdash (p,q)$ such that $p+q=n$, we consider the following homomorphism of algebras.

\begin{eqnarray}\label{eq:rhol}
 \rho_{\Lambda} : R[x_{1},\ldots,x_{n}] & \rightarrow & R[y_{1},\ldots,y_{r_1},y'_{1},\ldots,y'_{r_2}] \\ \nonumber
 f(x_{1},\ldots,x_{n}) & \mapsto & f( \underbrace{y_{1},\ldots,y_{1}}_{\lambda_{1}} ,\ldots, \underbrace{y_{r_1},\ldots,y_{r_1}}_{\lambda_{r_1}}, \underbrace{y'_{1},\ldots,y'_{1}}_{\lambda'_{1}} ,\ldots, \underbrace{y'_{r_2},\ldots,y'_{r_2}}_{\lambda'_{r_2}}).
 \end{eqnarray} where $y_{1},\ldots,y_{r_1},y'_{1},\ldots,y'_{r_2}$ are new indeterminates.\\
For two integers ${i,j} \in \{ 1,\ldots,n \}$ and $\pi \in \mathcal{S}_{1,\ldots,p} \times \mathcal{S}_{p+1, \ldots,n}$ such that $\pi(l)=l$ if $l\notin \{i,j \}$ and $\pi(i)=j$, then 
\begin{equation} \label{eq:BB}
 f^{\{i\}} -f^{\{j\}} = f^{\{i\}}-\pi(f^{\{i\}}) \in (x_i-x_j).
 \end{equation}
Therefore the polynomials systems $f^{\{1\}}, \ldots f^{\{p\}}$ and  $f^{\{p+1\}} \ldots f^{\{n\}}$ admit divided differences. From \cite[Lemma 2.1]{BuKa16} and \eqref{eq:GG}
for any subsets $\{i_1, \ldots, i_k \} \subset \{1,\ldots,p\}$, $\{j_1, \ldots, j_l\} \subset \{p+1,\ldots,n \}$ and $\phi=(\pi, \sigma) \in  \mathcal{S}_{1,\ldots,p} \times \mathcal{S}_{p+1, \ldots,n}$, we have 
\begin{equation}
\phi (f^{\{i_1, \ldots, i_k \}})= f^{ \{ \pi(i_1), \ldots, \pi(i_k) \}},\  \mbox{and}\  \phi (f^{\{j_1, \ldots, j_l \}})= F^{ \{ \sigma (j_1), \ldots, \sigma(j_l) \}}.
\end{equation}
Whenever $\rho_{\Lambda} (x_i)=\rho_{\Lambda}(x_j)$, by \eqref{eq:BB} we have 
$$ \rho_{\Lambda} (f^{\{i\}})= \rho_{\Lambda}(f^{\{j\}}).$$ 

So, for any integer $ i_1 \in \{ 1, \ldots, r_1 \}$ (respectively $ i_2 \in \{ 1, \ldots, r_2 \}$ ) we define the homogeneous polynomial
$$ f_{\Lambda}^{\{i_1\}} := \rho_{\Lambda} (f^{\{j_1\}}),\  (\  \mbox{respectively}\  f_{\Lambda}^{\{i_2\}} := \rho_{\Lambda} (F^{\{j_2\}}),$$
where $ j_1 \in \{1,\ldots, p \}$ such that $ \rho_{\Lambda}(x_{j_1})=y_{i_1},$ ( respectively $ j_2 \in \{p+1,\ldots, \} $ such that $ \rho_{\Lambda}(x_{j_2})=y_{i_2}$).\\
For  $I = \{i_1, \ldots, i_k \} \subset  \{1,\ldots, p \}$, define  $J = \{ j_1, \ldots, j_k \} \subset  \{1,\ldots, r_1 \}$
by the equality $ \rho_{\Lambda} (x_{i_r})= y_{j_r}$ forall $r \in \{ 1, \ldots, k\}$ ( respectively  
 $I' = \{i'_1, \ldots, i'_l \} \subset  \{p+1,\ldots, n \}$, define  $J' = \{ j'_1, \ldots, j'_l \} \subset  \{1,\ldots, r_2 \}$
by the equality $ \rho_{\Lambda} (x_{i'_r})= y_{j'_r}$ forall $r \in \{ 1, \ldots, l\}$ ). Then if $|I|=|J|$ (respectively $|I'|=|J'|$) we have
$$ \rho_{\Lambda}(f^{I})= \rho_{\Lambda} (f^{J}),\ (\ \mbox{respectively}\   \rho_{\Lambda}(f^{I'})= \rho_{\Lambda} (f^{J'}).$$

\subsection{The decomposition formula}

\begin{theorem}\label{thm:maintheorem} Assume that $n\geq 2$ and  assume a system of $n$ homogeneous polynomials $f^{\{1\}},\ldots, f^{\{p\}}, f^{\{p+1\}},\ldots, f^{\{n\}}$ in $R[x_1,\ldots,x_n]$ of the same degree $d$ equivariant with respect to the direct product of two symmetric groups  $\mathcal{S}_{\{1
,\ldots,p\}}\times\mathcal{S}_{\{p+1
,\ldots,n\}} $ with $1\leq p<n$. Let's put $q=n-p$, $\Lambda=(\lambda, \lambda')\vdash (p, q)$. 
. 

\noindent $\bullet$ If $p\leq d$ and $q\leq d$ then :
\begin{multline*}
\Res\left(f^{\{1\}},\ldots,f^{\{p\}},f^{\{p+1\}},\ldots,f^{\{n\}}\right)=\\ \pm
 \prod_{\substack{\Lambda\vdash (p,q)}}
\Res\left( f_{\Lambda}^{\{1\}}, f_{\Lambda}^{\{1,2\}}, \ldots, f_{\Lambda}^{\{1,2,\ldots,r_1\}},f_{\Lambda}^{\{p+1\}}, f_{\Lambda}^{\{p+1,p+2\}}, \ldots, f_{\Lambda}^{\{p+1,p+2,\ldots,p+r_2\}} 
\right)^{m_{\lambda}m_{\lambda'}}.
\end{multline*}

\noindent $\bullet$ If $p> d$ and $q\leq d$ then
\begin{multline*}
\Res\left(f^{\{1\}},\ldots,f^{\{p\}},f^{\{p+1\}},\ldots,f^{\{n\}}\right)=
 \pm \left(f^{\{1,\ldots,d+1\}}\right)^{\mu_1}\\ \times\prod_{\substack{\Lambda\vdash (p,q)\\ r_1\leq d}}
\Res\left( f_{\Lambda}^{\{1\}}, f_{\Lambda}^{\{1,2\}}, \ldots, f_{\Lambda}^{\{1,2,\ldots,r_1\}},f_{\Lambda}^{\{p+1\}}, f_{\Lambda}^{\{p+1,p+2\}}, \ldots, f_{\Lambda}^{\{p+1,p+2,\ldots,p+r_2\}} 
\right)^{m_{\lambda}m_{\lambda'}}.
\end{multline*}

\noindent $\bullet$ If $p\leq d$ and $q> d$ then
\begin{multline*}
\Res\left(f^{\{1\}},\ldots,f^{\{p\}},f^{\{p+1\}},\ldots,f^{\{n\}}\right)=
 \pm \left(f^{\{p+1,\ldots,p+1+d\}}\right)^{\mu_2}\\ \times\prod_{\substack{\Lambda\vdash (p,q)\\ r_2\leq d}}
\Res\left( f_{\Lambda}^{\{1\}}, f_{\Lambda}^{\{1,2\}}, \ldots, f_{\Lambda}^{\{1,2,\ldots,r_1\}},f_{\Lambda}^{\{p+1\}}, f_{\Lambda}^{\{p+1,p+2\}}, \ldots, f_{\Lambda}^{\{p+1,p+2,\ldots,p+r_2\}} 
\right)^{m_{\lambda}m_{\lambda'}}.
\end{multline*}

\noindent $\bullet$ If $p> d$ and $q> d$ then
\begin{multline*}
\Res\left(f^{\{1\}},\ldots,f^{\{p\}},f^{\{p+1\}},\ldots,f^{\{n\}}\right)=
 \pm \left(f^{\{1,\ldots,d+1\}}\right)^{\mu_1} \times \left(f^{\{p+1,\ldots,p+1+d\}}\right)^{\mu_2}\\ \times\prod_{\substack{\Lambda\vdash (p,q)\\ r_1\leq d,r_2\leq d}}
\Res\left( f_{\Lambda}^{\{1\}}, f_{\Lambda}^{\{1,2\}}, \ldots, f_{\Lambda}^{\{1,2,\ldots,r_1\}},f_{\Lambda}^{\{p+1\}}, f_{\Lambda}^{\{p+1,p+2\}}, \ldots, f_{\Lambda}^{\{p+1,p+2,\ldots,p+r_2\}} 
\right)^{m_{\lambda}m_{\lambda'}}.
\end{multline*}

where
\begin{equation*}
   \mu_1:=nd^{n-1}-\sum_{ \substack{\lambda\vdash p \\ r_1 \leq d}} m_{\lambda} 
  \left(\sum_{j=1}^{r_1}\dfrac{d(d-1)\cdots(d-r_1+1)\times d^q}{d-j+1}+d(d-1)\cdots(d-r_1+1)\times qd^{q-1} \right).
  \end{equation*}
  and
\begin{equation*}
   \mu_2:=nd^{n-1}-\sum_{ \substack{\lambda'\vdash q \\ r_2 \leq d}} m_{\lambda'} 
  \left(\sum_{j=1}^{r_2}\dfrac{d(d-1)\cdots(d-r_2+1)\times d^p}{d-j+1}+d(d-1)\cdots(d-r_2+1)\times pd^{p-1} \right).
  \end{equation*}

\end{theorem}

\paragraph*{\it {\bf Idea of the proof}}
The main idea of the proof is the same as the one in \cite{BuKa16}. In fact It is clear that the system $\{f^{\{1\}},\ldots, f^{\{p\}} \}$ is equivariant respect to  $\mathcal{S}_{\{1,\ldots,p\}} $ and the system $\{ f^{\{p+1\}},\ldots f^{\{n\}} \}$ is equivariant with respect to $\mathcal{S}_{\{p+1,\ldots,n\}}$. The proof goes on by splitting the resultant of  $f^{{\{i\}}}$'s into several factors by means of their divided differences associated to  $\mathcal{S}_{\{1,\ldots,p\}} $ and respectively. For the rest of the proof , we mimick  the proof of \cite[Theorem 3.3]{BuKa16}. Indeed this is a generalization of the proof of \cite[Theorem 3.3]{BuKa16}. \\
For the sake of the number of pages,the detailed discussions of the proof  will be published later in an extended version of this paper. \\

\begin{example}
Consider the following system  of $5$ homogeneous polynomials
 \\
\begin{equation*}
\begin{cases}
f^{\{1\}}=ax_1^2+bx_1^2+bx_1x_2+bx_1x_3+cx_1^2+cx_2^2+cx_3^2+x_4x_5\\ 
f^{\{2\}}=ax_2^2+bx_1x_2+bx_2^2+bx_2x_3+cx_1^2+cx_2^2+cx_3^2+x_4x_5\\
f^{\{3\}}=ax_3^2+bx_1x_3+bx_2x_3+bx_3^2+cx_1^2+cx_2^2+cx_3^2+x_4x_5\\
f^{\{4\}}=px_4^2+qx_5^2\\
f^{\{5\}}=px_5^2+qx_4^2.
\end{cases}
\end{equation*}\\
This system $\{f^{\{1\}}$, $f^{\{2\}}$, $f^{\{3\}}, f^{\{4\}}$, $f^{\{5\}} \}$ is not equivariant with respect to the symmetric group $\mathcal{S}_5$, then the  formula of \cite[Theorem 3.3]{BuKa16} cannot help to split the resultant of that polynomial system. But This system is equivariant to the direct product $\mathcal{S}_{\{1
,2,3\}}\times\mathcal{S}_{\{4,5\}} $. Then 
$f^{\{1\}}$, $f^{\{2\}}$, $f^{\{3\}}$  equivariant with  respect to $\mathcal{S}_{\{1
,2,3\}}$ and $f^{\{4\}}$, $f^{\{5\}}$ equivariant with  respect to $\mathcal{S}_{\{4,5\}}.$

\begin{multline*}
\Res\left(f^{\{1\}},f^{\{2\}},f^{\{3\}},f^{\{4\}},f^{\{5\}}\right)= \Big(f^{\{1,2,3\}}\Big)^{\mu_1}
\Res\left( f_{(3),(2)}^{\{1\}}, f_{(3),(2)}^{\{4\}}\right)^{m_{(3)}m_{(2)}} \\ \times \Res\left( f_{(2,1),(2)}^{\{1\}}, f_{(2,1),(2)}^{\{1,2\}} f_{(2,1),(2)}^{\{4\}}\right)^{m_{(2,1)}m_{(2)}} \times \Res\left( f_{(3),(1,1)}^{\{1\}}, f_{(3),(1,1)}^{\{4\}} f_{(3),(1,1)}^{\{4,5\}}\right)^{m_{(3)}m_{(1,1)}}\\ \times \Res\left( f_{(2,1),(1,1)}^{\{1\}},f_{(2,1),(1,1)}^{\{1,2\}} f_{(2,1),(1,1)}^{\{4\}} f_{(2,1),(1,1)}^{\{4,5\}}\right)^{m_{(2,1)}m_{(1,1)}}.
\end{multline*}
$f^{\{1\}}_{(3),(2)}=(a+3b+3c)x_1^2+x_4^2,\ 
f^{\{4\}}_{(3),(2)}=(p+q)x_4^2,\\
f^{\{1\}}_{(2,1),(2)}=(a+2b+2c)x_1^2+cx_2^2+bx_1x_2+x_4^2, \\
f^{\{1,2\}}_{(2,1),(2)}=(a+2b)x_1+(a+b)x_2,
f^{\{4\}}_{(2,1),(2)}=(p+q)x_4^2, \\
f^{\{1\}}_{(3),(1,1)}=(a+3b+3c)x_1^2+x_4x_5,
f^{\{4\}}_{(3),(1,1)}=px_4^2+qx_5^2, \\
f^{\{4,5\}}_{(3),(1,1)}=(p-q)x_4+(p-q)x_5, 
f^{\{1\}}_{(2,1),(1,1)}=(a+2b+2c)x_1^2+cx_2^2+bx_1x_2+x_4x_5, \\
f^{\{1,2\}}_{(2,1),(1,1)}=(a+2b)x_1+(a+b)x_2, 
f^{\{4\}}_{(2,1),(1,1)}=px_4^2+qx_5^2, \\ 
f^{\{4,5\}}_{(2,1),(1,1)}=(p-q)x_4+(p-q)x_5, f^{\{1,2,3\}}=a,\ \mu_1=8.$\\
we have  $\Res\left( f_{(3),(2)}^{\{1\}}, f_{(3),(2)}^{\{4\}}\right)=(3b+3c+a)^2(p+q)^2$
$$\Res\left( f_{(2,1),(2)}^{\{1\}}, f_{(2,1),(2)}^{\{1,2\}} f_{(2,1),(2)}^{\{4\}}\right)=(p+q)^2(a^3+3a^2b+3a^2c+2ab^2+8abc+6b^2c)^2 $$

$\Res\left( f_{(3),(1,1)}^{\{1\}}, f_{(3),(1,1)}^{\{4\}} f_{(3),(1,1)}^{\{4,5\}}\right)= (p+q)^2(p-q)^4(3b+3c+a)^2 $

$$ \Res\left( f_{(2,1),(1,1)}^{\{1\}},f_{(2,1),(1,1)}^{\{1,2\}}, f_{(2,1),(1,1)}^{\{4\}}, f_{(2,1),(1,1)}^{\{4,5\}}\right)= (p+q)^2(p-q)^4(a^3+3a^2b+3a^2c+2ab^2+8abc+6b^2c)^{12} .$$
$$
\Res\left(f^{\{1\}},f^{\{2\}},f^{\{3\}},f^{\{4\}}, f^{\{5\}} \right)= a^8(a^3+3a^2b+3a^2c+2ab^2+8abc+6b^2c)^{12}(p-q)^{16}(p+q)^{16}(3b+3c+a)^4.
$$
 
\end{example}

\section{Discriminant of a homogeneous polynomial invariant under direct product of symmetric groups }\label{sec:discriminant}

In this section, we will use the previous theorem  Theorem {\ref{thm:maintheorem}} to  develop  a decomposition formula for the discriminant of an invariant homogeneous polynomial under the action of  $\mathcal{S}_{\{1
,\ldots,p\}}\times\mathcal{S}_{\{p+1
,\ldots,n\}}.$ 
Let $f \in R[x_1,\ldots,x_p,x_{p+1},\ldots,x_n]$ of degree $d$ be a homogeneous polynomial that is invariant under direct product  $\mathcal{S}_{\{1
,\ldots,p\}}\times\mathcal{S}_{\{p+1
,\ldots,n\}} $ of symmetric groups with $1\leq p< n$.

For all $\sigma_1\in \mathcal{S}_{\{1
,\ldots,p\}}$ and for all $\sigma_2\in \mathcal{S}_{\{p+1
,\ldots,n\}}$, we have :

\begin{eqnarray}
(\sigma_1,\sigma_2)\Big(f \Big)(x_1,\ldots,x_p,x_{p+1},\ldots,x_n) &=&f(x_{\sigma_1(1)},\ldots x_{\sigma_1(p)},x_{\sigma_2(p+1)},\ldots,x_{\sigma_2(n)}). \\
 &=& f(x_1,\ldots,x_p,x_{p+1},\ldots,x_n)
\end{eqnarray}

We  will denote the partial derivatives of $F$ by
$$f^{\{i\}}(x_1,\ldots,x_p,x_{p+1},\ldots,x_n):=\frac{\partial f}{\partial x_{i}} (x_1,\ldots,x_p,x_{p+1},\ldots,x_n), \ i=1,\ldots n .$$ 
The discriminant of $F$ is defined by the equality
\begin{equation}\label{eq:discressym}
d^{a(n,d)}\Disc(f)=\Res\left(  f^{\{1\}},f^{\{2\}},\ldots,f^{\{n\}} \right) \in \UU
\end{equation} where 
$$a(n,d):=\frac{(d-1)^{n}-(-1)^{n}}{d} \in \ZZ.$$
and that it is homogeneous of degree $n(d-1)^{n-1}$.

\begin{lemma}
The set $\{ f^{\{1\}},f^{\{2\}},\ldots,f^{\{p\}},f^{\{p+1\}},\ldots,f^{\{n\}} \}$ of partial derivatives of a  $\mathcal{S}_{\{1,\ldots,p\}}\times\mathcal{S}_{\{p+1,\ldots,n\}} $-invariant homogeneous polynomial $f$ is is an equivariant polynomial system with respect to $\mathcal{S}_{\{1,\ldots,p\}}\times\mathcal{S}_{\{p+1,\ldots,n\}} $.
\end{lemma}
\begin{proof}
We will use the canonical inclusions  $ \mathcal{S}_{\{1,\ldots,p\}} \to \mathcal{S}_{\{1,\ldots,p\}}\times\mathcal{S}_{\{p+1,\ldots,n\}}, \sigma_1 \mapsto (\sigma_1, e_2)$ and  $ \mathcal{S}_{\{p+1,\ldots,n\}} \to \mathcal{S}_{\{1,\ldots,p\}}\times\mathcal{S}_{\{p+1,\ldots,n\}}, \sigma_2 \mapsto (e_1,\sigma_2)$ where $e_1, e_2$ are unit elements of $ \mathcal{S}_{\{1,\ldots,p\}}$ and $\mathcal{S}_{\{p+1,\ldots,n\}}$ respectively.
For all $i\in \{1,\ldots,p\}$ and $\sigma_1 \in \mathcal{S}_{\{1,\ldots,p\}},$ we have $\sigma_1\Big(f^{\{i\}}\Big)=\sigma_1\Big(\frac{\partial f}{\partial x_{i}}\Big)=\frac{\partial (\sigma_1f)}{\partial x_{\sigma_1(i)}}=\frac{\partial f}{\partial x_{\sigma_1(i)}} =f^{\{ {\sigma_1(i)} \}}.$ For all $j\in \{p+1,\ldots,n\}$ and $\sigma_2 \in \mathcal{S}_{\{p+1,\ldots,n\}},$ , $\sigma_2\Big(f^{\{j\}}\Big)=\sigma_2\Big(\frac{\partial f}{\partial x_{j}}\Big)=\frac{\partial (\sigma_2 f)}{\partial x_{\sigma_2(j)}}=\frac{\partial f}{\partial x_{\sigma_2(j)}} =f^{\{ {\sigma_2(j)} \}}.$ For all $k \in \{1,\ldots,p\}$ $\sigma= (\sigma_1, \sigma_2) \in \mathcal{S}_{\{1,\ldots,p\}}\times\mathcal{S}_{\{p+1,\ldots,n\}}$, we have $\sigma \Big( f^{\{k\}} \Big)= \begin{cases}
\sigma_1 \Big( f^{\{k\}} \Big)= f^{\{ {\sigma_1(k)} \}}\  \mbox{if}\  k\in \{1,\ldots,p\} \\
\sigma_2 \Big( f^{\{k\}} \Big)= f^{\{ {\sigma_2(k)} \}}\  \mbox{if}\  k\in \{p+1,\ldots,n\}.
\end{cases}$
Hence $\sigma \Big( f^{\{k\}} \Big) \in  \{ f^{\{1\}},\ldots ,f^{\{n\}} \}$, for all  $k \in \{1,\ldots,p\}$ $\sigma= (\sigma_1, \sigma_2) \in \mathcal{S}_{\{1,\ldots,p\}}\times\mathcal{S}_{\{p+1,\ldots,n\}}$. and the set of partial derivative of $f$ form  an equivariant polynomial system with respect to $\mathcal{S}_{\{1,\ldots,p\}}\times\mathcal{S}_{\{p+1,\ldots,n\}} $.
\end{proof}

As a consequence of this lemma, Theorem \ref{thm:maintheorem} can be applied in order to decompose the resultant of the polynomials $f^{\{1\}},f^{\{2\}},\ldots,f^{\{p\}},f^{\{p+1\}},\ldots f^{\{n\}}$ and hence, by \eqref{eq:discressym}, to decompose the discriminant of the $\mathcal{S}_{\{1,\ldots,p\}}\times\mathcal{S}_{\{p+1,\ldots,n\}} $-invariant polynomial $f$. 

\begin{theorem}\label{thm:discmaintheorem} Assume that $n\geq 2$ and $d\geq 2$. With the above notation, the following equalities hold.

\noindent $\bullet$ If $p< d$ and $q< d$ then :
\begin{multline*}
d^{a(n,d)}\Disc\left(f\right)=\\
 \prod_{\substack{\Lambda\vdash (p,q)}}
\Res\left( f_{(\lambda,\lambda')}^{\{1\}}, f_{(\lambda,\lambda')}^{\{1,2\}}, \ldots, f_{(\lambda,\lambda')}^{\{1,2,\ldots,r_1\}},f_{(\lambda,\lambda')}^{\{p+1\}}, f_{(\lambda,\lambda')}^{\{p+1,p+2\}}, \ldots, f_{(\lambda,\lambda')}^{\{p+1,p+2,\ldots,p+r_2\}} 
\right)^{m_{\lambda}m_{\lambda'}}.
\end{multline*}

\noindent $\bullet$ If $p\geqslant d$ and $q< d$ then :
\begin{multline*}
d^{a(n,d)}\Disc\left(f\right)=
  \left(f^{\{1,\ldots,d\}}\right)^{\mu_1}\\ \times\prod_{\substack{\Lambda\vdash (p,q)\\ r_1< d}}
\Res\left( f_{(\lambda,\lambda')}^{\{1\}}, f_{(\lambda,\lambda')}^{\{1,2\}}, \ldots, f_{(\lambda,\lambda')}^{\{1,2,\ldots,r_1\}},f_{(\lambda,\lambda')}^{\{p+1\}}, f_{(\lambda,\lambda')}^{\{p+1,p+2\}}, \ldots, f_{(\lambda,\lambda')}^{\{p+1,p+2,\ldots,p+r_2\}} 
\right)^{m_{\lambda}m_{\lambda'}}.
\end{multline*}

\noindent $\bullet$ If $p< d$ and $q\geqslant d$ then :
\begin{multline*}
d^{a(n,d)}\Disc\left(f\right)=
  \left(f^{\{p+1,\ldots,p+d\}}\right)^{\mu_2}\\ \times\prod_{\substack{\Lambda\vdash (p,q)\\ \ r_2< d}}
\Res\left( f_{(\lambda,\lambda')}^{\{1\}}, f_{(\lambda,\lambda')}^{\{1,2\}}, \ldots, f_{(\lambda,\lambda')}^{\{1,2,\ldots,r_1\}},f_{(\lambda,\lambda')}^{\{p+1\}}, f_{(\lambda,\lambda')}^{\{p+1,p+2\}}, \ldots, f_{(\lambda,\lambda')}^{\{p+1,p+2,\ldots,p+r_2\}} 
\right)^{m_{\lambda}m_{\lambda'}}.
\end{multline*}

\noindent $\bullet$ If $p\geqslant d$ and $q\geqslant d$ then :
\begin{multline*}
d^{a(n,d)}\Disc\left(F\right)=
  \left(f^{\{1,\ldots,d\}}\right)^{\mu_1}\times \left(f^{\{p+1,\ldots,p+d\}}\right)^{\mu_2}\\ \times\prod_{\substack{\Lambda\vdash (p,q)\\ r_1< d,\ r_2< d}}
\Res\left( f_{(\lambda,\lambda')}^{\{1\}}, f_{(\lambda,\lambda')}^{\{1,2\}}, \ldots, f_{(\lambda,\lambda')}^{\{1,2,\ldots,r_1\}},F_{(\lambda,\lambda')}^{\{p+1\}}, f_{(\lambda,\lambda')}^{\{p+1,p+2\}}, \ldots, f_{(\lambda,\lambda')}^{\{p+1,p+2,\ldots,p+r_2\}} 
\right)^{m_{\lambda}m_{\lambda'}}.
\end{multline*}

where
 \begin{equation*}
    \mu_1:=n(d-1)^{n-1}-\sum_{ \substack{\lambda\vdash p \\ r_1 < d}} m_{\lambda} 
   \left(\sum_{j=1}^{r_1}\dfrac{(d-1)\cdots(d-r_1)\times (d-1)^q}{d-j}+(d-1)\cdots(d-r_1)\times q(d-1)^{q-1} \right).
   \end{equation*}
   and
 \begin{equation*}
    \mu_2:=n(d-1)^{n-1}-\sum_{ \substack{\lambda'\vdash q \\ r_2 < d}} m_{\lambda'} 
   \left(\sum_{j=1}^{r_2}\dfrac{(d-1)\cdots(d-r_2)\times (d-1)^p}{d-j}+(d-1)\cdots(d-r_2)\times p(d-1)^{p-1} \right).
   \end{equation*}
\end{theorem}

\begin{proof}
These formulas are obtained by specialization of the formulas given in Theorem \ref{thm:maintheorem} with the difference that the polynomials $f^{\{i\}}$, $i=1,\ldots,n$ are of degree $d-1$ in our setting (and not of degree $d$ as in Theorem \ref{thm:maintheorem}).
\end{proof}
\begin{example}
Consider  a homogeneous polynomial of degree 4.
\begin{equation*}
f:=ax_1^4+bx_1^2x_2^2+ax_2^4+cx_3^4+x_3x_4^3+x_3^3x_4+cx_4^4
\end{equation*}
$F$ is not symmetric polynomial but $F$ an invariant polynomial  under the action of the direct product  $\mathcal{S}_{\{1
,2\}}\times\mathcal{S}_{\{3,4\}} $

Its partial derivatives are :
$\begin{cases}
f^{\{1\}}=4ax_1^3+2bx_1x_2^2\\
f^{\{2\}}=4ax_2^3+2bx_1^2x_2\\
f^{\{3\}}=4cx_3^3+x_4^3+3x_3^2x_4\\
f^{\{4\}}=4cx_4^3+x_3^3+3x_3x_4^2
\end{cases}$

$f^{\{1\}}$, $f^{\{2\}}$ equivariant with  respect to $\mathcal{S}_{\{1
,2\}}$ and $f^{\{3\}}$, $f^{\{4\}}$ equivariant with  respect to $\mathcal{S}_{\{3
,4\}}$ 

The formula given in Theorem \ref{thm:discmaintheorem} shows that
\begin{multline*}
4^{\frac{3^{4}-(-1)^{4}}{4}}\Disc(f)=
\Res\left( f_{(2),(2)}^{\{1\}}, f_{(2),(2)}^{\{3\}}\right) \times \Res\left( f_{(1,1),(2)}^{\{1\}}, f_{(1,1),(2)}^{\{1,2\}} f_{(1,1),(2)}^{\{3\}}\right)\\ \times \Res\left( f_{(2),(1,1)}^{\{1\}}, f_{(2),(1,1)}^{\{3\}} f_{(2),(1,1)}^{\{3,4\}}\right)\times \Res\left( f_{(1,1),(1,1)}^{\{1\}},f_{(1,1),(1,1)}^{\{1,2\}} f_{(1,1),(1,1)}^{\{3\}} f_{(1,1),(1,1)}^{\{3,4\}}\right).
\end{multline*}
we have

$\Res\left( f_{(2),(2)}^{\{1\}}, f_{(2),(2)}^{\{3\}}\right)=512(2a+b)^3(c+1)^3$

$\Res\left( f_{(1,1),(2)}^{\{1\}}, f_{(1,1),(2)}^{\{1,2\}} f_{(1,1),(2)}^{\{3\}}\right)=8589934592a^6(2a+b)^3(2a-b)^6(c+1)^6 $

$\Res\left( f_{(2),(1,1)}^{\{1\}}, f_{(2),(1,1)}^{\{3\}} f_{(2),(1,1)}^{\{3,4\}}\right)=262144(2a+b)^6(c-1)^3(8c^2+1)^6 $

$$ \Res\left( f_{(1,1),(1,1)}^{\{1\}},f_{(1,1),(1,1)}^{\{1,2\}} f_{(1,1),(1,1)}^{\{3\}} f_{(1,1),(1,1)}^{\{3,4\}}\right)=73786976294838206464a^{12}(2a+b)^6(2a-b)^{12}(c-1)^6(8c^2+1)^{12} $$

\begin{center}
 $$ \Disc(f) = 77371252455336267181195264a^{18}(c-1)^9(c+1)^9(8c^2+1)^{18}(2a-b)^{18}(2a+b)^{18}.$$
\end{center}

\end{example}

\section*{Acknowledgments}
 Gratitude is expressed to the projects African Centre of Excellence in Mathematical Sciences, Informatics and Applications (ACE-MSIA)


\end{document}